\documentclass[reqno]{amsart}

\usepackage{graphicx} \usepackage{mathtools}
\usepackage{calrsfs} \usepackage{microtype} \usepackage{enumitem}
\usepackage{mathrsfs}
\usepackage[utf8]{inputenc}
\usepackage[T1]{fontenc}
\usepackage{textcomp} 

\usepackage{amsmath}
\usepackage{amsfonts} \usepackage{amsmath} \usepackage{amsrefs}
\usepackage{amssymb} \usepackage{amstext} \usepackage{amsthm}

\usepackage{cancel} \usepackage{caption} \usepackage{enumitem}
\usepackage{hyperref} \usepackage{cleveref} \usepackage{subcaption}
\usepackage{xcolor}
\usepackage[english]{babel}
\usepackage{comment}

\theoremstyle{plain}

\usepackage{blkarray}
\usepackage{multirow}
\usepackage{mathtools}
\usepackage{graphicx}

\newtheorem{theorem}{Theorem}[section]

\newtheorem{proposition}[theorem]{Proposition}

\theoremstyle{definition} 

\theoremstyle{remark} \newtheorem{remark}[theorem]{Remark}

\newtheorem{prop}[theorem]{Proposition}

  \def\cC{\mathcal C}
 
 \newcommand{\cM}{\mathcal M}

\def\cH{\mathcal H}  
 
 \newcommand{\PG}{\mathrm{PG}}

\newcommand{\GF}[1]{\mathbb{F}_{#1}}

\title{Quasi-Hermitian varieties and their Barlotti-Cofman Representation}
\author{Angela Aguglia \and Viola Siconolfi}
\begin{document}

\maketitle

\begin{abstract}
Quasi-Hermitian varieties arise as higher-dimensional generalizations of non-classical unitals, including the Buekenhout--Metz (BM) and Buekenhout--Tits (BT) families. After reviewing known constructions and structural properties, we determine explicitly the Barlotti--Cofman (BC) representation of BM and BT quasi-Hermitian varieties in $\PG(3,q^2)$ inside $\PG(6,q)$. 

We show that BM quasi-Hermitian varieties correspond to quadratic cones with hyperbolic base, whereas in the BT case the associated hypersurfaces admit a cone-like description and are non-quadratic. We also describe the configuration of spread elements in the section at infinity.

These results provide a geometric interpretation of the non-classical nature of BM and BT varieties within the BC framework.

\end{abstract}

\smallskip


\maketitle

\section{Introduction}

Unitals in finite projective planes of square order have been widely studied for
their rich combinatorial structure and their connections with incidence geometry
and coding theory; see, for instance, \cite{E2}, \cite{BarE},  \cite{JH}, \cite{HKT}.
A {\em unital} embedded in the Desarguesian plane $\PG(2,q^2)$ is a set of $q^3+1$ points such that every line meets it in either $1$ or $q+1$ points.
In  $\PG(2,q^2)$, the classical example is the Hermitian
unital arising from a non-degenerate Hermitian curve (cf.\ Segre \cite{Segre1965}).
Beyond the classical case, the main known infinite families of embedded
 unitals in $\PG(2,q^2)$ are the Buekenhout--Metz (BM) unitals
\cite{BE,E} (including Hermitian unitals) and, in even characteristic, the Buekenhout--Tits (BT) unitals
(see \cite{BarE}), both obtained via  the André/Bruck–Bose  model of $\PG(2,q^2)$ in $\PG(4,q)$.

A natural problem is to lift these plane constructions to higher-dimensional
spaces while retaining the characteristic two-intersection pattern with
hyperplanes (see Section~\ref{sec:QPspaces}). Quasi-Hermitian varieties provide a
framework for such an extension. These are point sets in $\PG(r,q^2)$ ($r\ge 2$)
having the same hyperplane intersection numbers as a non-singular Hermitian
variety, while allowing genuinely non-classical behaviour.
Since it has only two possible intersection sizes with hyperplanes, a quasi-Hermitian variety in $\mathrm{PG}(r,q^2)$ is a \emph{two-character set}.

Two-character sets have wide-ranging applications as
    they give rise to strongly regular graphs and
    two weight linear codes (see Delsarte \cite{D}, Calderbank and Kantor \cite{CK}).

Among quasi-Hermitian varieties, BM and BT quasi-Hermitian varieties can be regarded as
higher-dimensional analogues of the BM and BT unitals, since for $r=2$ they
reduce to these familiar non-Hermitian unitals.

 This paper has two main goals. First, we provide a  survey on
quasi-Hermitian varieties, reviewing their definition, known constructions,
and recent developments. Second, we obtain new results on the
Barlotti--Cofman (BC) representation of BM and BT quasi-Hermitian varieties.

Sections~\ref{sec:QPspaces} and~\ref{sec:QHvarieties} summarize background on
quasi-polar spaces, classical and non-classical constructions, and related
recent work. A key tool throughout is the Barlotti--Cofman representation of
$\PG(r,q^2)$ inside $\PG(2r,q)$ via a Desarguesian spread.
Section~\ref{sec:BCrep} recalls this representation and fixes our notation.
Working over $\mathbb{F}_q$ allows one to reinterpret varieties defined over
$\mathbb{F}_{q^2}$ as configurations in a larger projective space, where the
interaction with the hyperplane carrying the spread encodes the structure at
infinity.

Sections~\ref{sec:odd},~\ref{sec:even} and~\ref{sec:BT} contain the main new
results. There we determine explicitly the BC representation of BM and BT
quasi-Hermitian varieties in $\PG(3,q^2)$, viewed inside $\PG(6,q)$.

For BM quasi-Hermitian varieties, we show that their BC representation is the
affine part of a quadratic cone whose base is a non-degenerate hyperbolic quadric.
We describe the section at infinity and determine precisely the spread elements
contained in it; see Theorems~\ref{thm:babodd} and~\ref{thm:mabodd} (odd
characteristic) and Theorems~\ref{thm:babeven} and~\ref{thm:mabeven} (even
characteristic).

In contrast, BT quasi-Hermitian varieties exhibit a different behaviour: their BC representation yields a non-quadratic hypersurface in $\PG(6,q)$ admitting a cone-like description.
We explicitly determine the associated configuration of spread elements in the
section at infinity; see Theorems~\ref{thm:BT-quasiHermitian-cone}
and~\ref{thm:BTH}.

Beyond providing a coordinate model, our results reveal a structural
dichotomy in dimension three: BM varieties give rise to quadratic cones, whereas BT varieties lead to non-quadratic hypersurfaces which can be described as unions of lines through a fixed point. Moreover, the induced configuration
of spread elements at infinity encodes subtle arithmetic information (such as
$q\bmod 4$ in the BM case and $e\bmod 4$ in the BT case), suggesting that the BC
model may offer intrinsic $\mathbb{F}_q$-linear signatures for distinguishing,and
possibly classifying, non-classical quasi-Hermitian varieties.

The restriction to $r=3$ is motivated by the fact that already in this dimension
the interaction between the affine cone and the Desarguesian spread exhibits
non-trivial behaviour. Extending the analysis to higher dimensions raises further
combinatorial and algebraic complexities, which we leave for future work.

\section{Quasi-Polar Spaces and Historical Background}\label{sec:QPspaces}

A \emph{quasi-polar space} is a set of points $\mathcal{Q}$ in $\mathrm{PG}(r,q)$, where $r \ge 2$ and $q$ is a prime power, such that the intersection sizes with hyperplanes match those of a non-degenerate classical polar space $\mathscr{P}$ embedded in $\mathrm{PG}(r,q)$. 

The concept of quasi-polar spaces generalizes classical polar spaces by relaxing algebraic conditions while preserving combinatorial intersection properties. This idea traces back to Segre (1954), who defined an oval in a finite projective plane as a combinatorial abstraction of a conic in $\mathrm{PG}(2,q)$.

Quasi-quadrics were introduced by De Clerck, Hamilton, O'Keefe, and Penttila in \cite{DHOP}, whereas the analogous concept for Hermitian varieties, \emph{quasi-Hermitian varieties}, was formally introduced by De Winter and Schillewaert in \cite{DS}.

Here we focus on the Hermitian case.
A \emph{non-singular Hermitian variety} $\mathcal{H}(r,q^2)$ in $\mathrm{PG}(r,q^2)$ is defined as the set of absolute points of a non-degenerate unitary polarity $\rho$:
\[
\mathcal{H}(r,q^2) = \{ P \in \mathrm{PG}(r,q^2) \mid P \in P^\rho \},
\]
where $P^\rho$ denotes the polar hyperplane of $P$ under $\rho$.

Fixing a projective frame of $\mathrm{PG}(r,q^2)$ where $(X_0,X_1,\ldots,X_r)$ denote
 homogeneous coordinates for projective points, a non-singular Hermitian variety $\mathcal{H}(r,q^2)$ is a hypersurface with equation
\[
(X_0^q,\ldots,X_r^q) \, H \, (X_0,\ldots,X_r)^T = 0,
\]
where $H$ is a non-singular Hermitian $(r+1)\times(r+1)$ matrix.

Any non-singular Hermitian variety in $\mathrm{PG}(r,q^2)$ can be mapped to any other via a projectivity. In the plane, the non-singular Hermitian curve $\mathcal{H}(2,q^2)$ is known as the \emph{classical or  Hermitian  unital}. 

A point set $\mathcal{S}$ of $\mathrm{PG}(r,q^2)$ is a \emph{quasi-Hermitian variety} if it meets each hyperplane in one of the following two sizes:
\[
|\mathcal{H}(r-1,q^2)| = \frac{(q^r+(-1)^{r-1})(q^{r-1}-(-1)^{r-1})}{q^2-1}, \quad \text{or}
\]
\[
|P_0\mathcal{H}(r-2,q^2)| = \frac{(q^r+(-1)^{r-1})(q^{r-1}-(-1)^{r-1})}{q^2-1} + (-1)^{r-1}q^{r-1}.
\]

The non-singular Hermitian variety $\mathcal{H}(r,q^2)$ is itself a quasi-Hermitian variety, called the \emph{classical quasi-Hermitian variety}, but non classical examples can occur.

In the next section we review constructions that produce non-classical quasi-Hermitian varieties, including higher-dimensional analogues of the Buekenhout--Metz and Buekenhout--Tits unitals, which to date are the only known families of non classical unitals embedded in the Desarguesian plane.

\section{Quasi-Hermitian Varieties and Known Constructions}\label{sec:QHvarieties}

In this section we recall the main constructions and structural
results that will be relevant for the BC analysis carried out
in Sections~5--7. In particular, we focus on the BM and BT
families and on their known equivalence and automorphism properties.

The first nontrivial constructions of quasi-Hermitian varieties were introduced by De Winter and Schillewaert \cite{DS}
using  the \emph{pivoting method} (further developed in \cite{SchillewaertVandeVoorde2022}): a powerful tool for constructing quasi-Hermitian varieties from classical Hermitian varieties. The idea is to start with a non-degenerate Hermitian variety $\cH \subset \mathrm{PG}(r,q^2)$ and modify the configuration of its tangent hyperplanes while preserving the two-character property.

More precisely, one selects a set of tangent hyperplanes to $\cH$ and replaces each of them by another hyperplane through the same pivot point, ensuring that the intersection pattern with the remaining points is adjusted in a controlled way. This operation, called \emph{pivoting}, produces a new point set with the same intersection numbers with hyperplanes as $\cH$, but which is not projectively equivalent to $\cH$.

The pivoting technique can be applied iteratively (\emph{multi-pivoting}), generating infinite families of quasi-Hermitian varieties.

Beyond pivoting, more global approaches have been developed. Aguglia, Cossidente, and Korchmáros \cite{ACK,AA}, introduced families known as \emph{BM} (Buekenhout--Metz) and \emph{BT} (Buekenhout--Tits) quasi-Hermitian varieties, obtained by deforming Hermitian varieties through birational transformations and non-standard incidence models. These constructions generalize the  known non-Hermitian unitals in \(\mathrm{PG}(2,q^2)\) to higher dimensions, providing rich examples with large automorphism groups (see Subsection  \ref{Var1}).

In \cite{PF}, a geometric method to construct two-character sets in $\mathrm{PG}(r,q^2)$, which includes quasi-Hermitian varieties as a special case was introduced. The construction works for $r \ge 3$ and $q$ both odd. 

The approach starts from a Baer subgeometry $\Sigma \cong \mathrm{PG}(r,q)$ embedded in $\mathrm{PG}(r,q^2)$ and a non-degenerate quadric $Q_0$ contained in $\Sigma$. The key idea is to consider the set of all extended lines of $\mathrm{PG}(r,q^2)$ that meet $\Sigma$ in $q+1$ points and intersect $Q_0$ in either one point or $q+1$ points. The union of these lines forms a point set $B$ with the same intersection numbers with hyperplanes as a non-degenerate Hermitian variety, but which is not projectively equivalent to it. This makes $B$ a non classical quasi-Hermitian variety.

Recent constructions of quasi-Hermitian varieties in $\PG(3,q^2)$
have been obtained via geometric methods.
Lavrauw, Lia and Pavese \cite{LLP} constructed new examples
arising from the geometry of the Hermitian Veronese curve,
exploiting the action of subgroups of $PGL(4,q^2)$
that preserve both a Hermitian surface and a hyperbolic quadric.
These constructions yield non classical quasi-Hermitian surfaces.

Further examples were introduced by Lia and Sheekey  \cite{LS}
using tensorial constructions and commuting polarities.
Their approach produces additional families of two-character sets
in $\PG(3,q^2)$, again not projectively equivalent to previously
known examples.
We refer to the cited papers for full details.

Among these examples, the Buekenhout--Metz (BM)
and Buekenhout--Tits (BT) families play a central role,
and will be the focus of our BC analysis.

\subsection{ BM  and BT  quasi-Hermitian varieties}
\label{Var1}

We now present two fundamental families of quasi-Hermitian varieties in finite projective spaces of arbitrary dimension, generalizing the well-known non-classical unitals in \(\mathrm{PG}(2,q^2)\).

\subsubsection{  BM quasi-Hermitian varieties}
\label{ssec:BM}

In \cite{ACK} the authors introduced a nonstandard model $\Pi$ of $\mathrm{PG}(r,q^2)$ to construct new quasi-Hermitian varieties. This approach globally modifies the incidence structure of $\mathrm{PG}(r,q^2)$ via a quadratic transformation, rather than by local pivoting.

 $\Pi$ is constructed as follows. Fix a projective frame with homogeneous coordinates $(X_0,X_1,\dots,X_r)$ and consider the affine space $\mathrm{AG}(r,q^2)$ with infinite hyperplane
$
\Sigma_{\infty}: X_0 = 0.
$
Then, $AG(r,q^2)$ has affine coordinates $x_1, x_2, \dots, x_r$ where
$x_i = X_i / X_0$ for $i \in \{1,\dots,r\}$.
Let
$P_{\infty} = (0,0,\dots,0,1)$.
 Consider the incidence structure $\mathcal{I}$   whose
points are the points of $AG(r,q^2)$ and whose hyperplanes are the hyperplanes
through  $P_\infty = (0,0,\dots,0,1)$ together with the
 quadrics of the form
    \[
    Q_a(m,d):\; x_r = a(x_1^2+\cdots+x_{r-1}^2)+m_1x_1+\cdots+m_{r-1}x_{r-1}+d,
    \]
    where $a\in\mathbb{F}_{q^2}^*$, $m\in\mathbb{F}_{q^2}^{r-1}$, $d\in\mathbb{F}_{q^2}$.

The birational map
\[
\varphi:\;(x_1,\dots,x_{r-1},x_r)\mapsto (x_1,\dots,x_{r-1},x_r - a(x_1^2+\cdots+x_{r-1}^2)).
\]
induces an isomorphism between  $AG(r,q^2)$ and $\mathcal{I}$. Completing $\mathcal{I}$ with its points at infinity in the usual way gives a projective space $\Pi$ isomorphic to $\PG(r, q^2)$.

Under this model $\Pi$, the non-singular Hermitian variety $\cH(r,q^2)$ with affine equation
\[
x_r^q - x_r = (b^q-b)(x_1^{q+1}+\dots+x_{r-1}^{q+1}), \qquad b\in \mathbb{F}_{q^2}\setminus \mathbb{F}_q,
\]
becomes a quasi-Hermitian variety, that is it has the same intersection numbers with the hyperplanes of the new model as a classical Hermitian variety,
 provided the following conditions on $a,b$ hold:

\begin{equation}\label{cond:BM}
\begin{cases}
\text{For $q$ odd:}\\
\quad \text{if $r$ odd: } 4a^{q+1}+(b^q-b)^2 \neq 0,\\
\quad \text{if $r$ even: } 4a^{q+1}+(b^q-b)^2 \text{ is a non-square in }\mathbb{F}_q;\\[6pt]
\text{For $q$ even:}\\
\quad \text{if $r$ odd: no conditions,}\\
\quad \text{if $r$ even: } \mathrm{Tr}\bigl(a^{q+1}/(b^q+b)^2\bigr)=0,
\end{cases}
\end{equation}
where $\mathrm{Tr}$ denotes the trace of $\mathbb{F}_{q^2} \to \mathbb{F}_{q}$.

\noindent
We define $\mathcal{F}$ to be the cone  of $\Sigma_{\infty}$ satisfying:
\begin{equation}\label{eq:F}
\begin{cases}
X_0=0\\
X_1^{q+1}+X_2^{q+1}+\ldots+ X_{r-1}^{q+1} =0
\end{cases}
\end{equation}
 and 
$\mathcal{B}_{a,b}$ the variety with affine equation 
\[
x_r^q - x_r + a^q(x_1^{2q}+\dots+x_{r-1}^{2q}) - a(x_1^2+\dots+x_{r-1}^2) = (b^q-b)(x_1^{q+1}+\dots+x_{r-1}^{q+1}).
\]
As $\mathcal{B}_{a,b}$ and $\mathcal{H}(r,q^2)$ are birationally equivalent under $\varphi$, it turns out that the pointset  $\mathcal{M}_{a,b}:=(\mathcal{B}_{ab}\setminus \Sigma_{\infty}) \cup \mathcal{F}$ is a non-classical quasi-Hermitian variety of $\PG(r,q^2)$ under  conditions~\eqref{cond:BM}.

For $r=2$, $\mathcal{M}_{a,b}$ is a non-classical Buekenhout–Metz unital, that is,  not projectively equivalent to the Hermitian curve.

\subsubsection{  BT quasi-Hermitian varieties}\label{sssec:BT}

For $q = 2^e$ with $e>1$ odd, Aguglia~\cite{AA} introduced a second non-standard model of $\mathrm{PG}(r,q^2)$ leading to a new family of quasi-Hermitian varieties. This construction uses a birational transformation and an automorphism of $\mathbb{F}_q$ to globally modify the incidence structure as follows.

Choose $\varepsilon \in \mathbb{F}_{q^2}\setminus \mathbb{F}_q$ such that
\[
\varepsilon^2 + \varepsilon + \delta = 0, \qquad \delta \in \mathbb{F}_q,\ \operatorname{tr}(\delta)=1,
\]
where $\operatorname{tr}$ denotes the trace map $\mathbb{F}_q \to \mathbb{F}_2$.
Consider the automorphism
\[
\sigma : x \mapsto x^{2^{\frac{e+1}{2}}}.
\]

For any vector $m = (m_1, \ldots, m_{r-1}, d) \in \mathrm{GF}(q^{2})^{\,r}$,
let $D(m)$ denote the algebraic hypersurface defined by
\begin{equation}\label{eq:2.1}
x_r =
\Delta_{\varepsilon}(x_1)
+ \cdots
+ \Delta_{\varepsilon}(x_{r-1})
+ m_1 x_1
+ \cdots
+ m_{r-1} x_{r-1}
+ d.
\end{equation}
where 
\[
\Delta_{\varepsilon}(x)
= \varepsilon x^{q(\sigma+2)}
  + \left( \varepsilon^{\sigma} + \varepsilon^{\sigma+2} \right) x^{q\sigma + 2}
  + x^{\sigma}
  + (1 + \varepsilon)x^{2}.
\]

Define $\mathcal{I}'$ as the incidence structure whose points are those of $\mathrm{AG}(r,q^2)$ and whose hyperplanes are the hyperplanes through $P_\infty = (0,0,\dots,0,1)$ together with the hypersurfaces $D(m)$, where $m$ ranges over $\mathbb{F}_{q^2}^r$.
The incidence structure $\mathcal{I}'$ is an affine space isomorphic to $\mathrm{AG}(r,q^2)$.

The birational transformation $\phi$ given by
\begin{equation}
\phi : (x_1,\dots,x_{r-1},x_r) \mapsto (x_1,\dots,x_{r-1},x_r + \Delta_{\varepsilon}(x_1)+\dots+\Delta_{\varepsilon}(x_{r-1}))
\tag{2.2}
\end{equation}
determines an isomorphism $\mathcal{I}' \cong \mathrm{AG}(r,q^2)$.
Completing $\mathcal{I}'$ with its points at infinity in the usual way gives a projective space $\Pi'$ isomorphic to $\mathrm{PG}(r,q^2)$.

\medskip
In this model $\Pi'$, the non-singular Hermitian variety $\mathcal{H}(r,q^2)$ with affine equation
\[
x_r^q - x_r = x_1^{q+1}+\dots+x_{r-1}^{q+1}, \qquad b\in \mathbb{F}_{q^2}\setminus \mathbb{F}_q,
\]
is a quasi-Hermitian variety.

As a consequence, for $x \in \mathbb{F}_{q^2}$, set
\[
\Gamma_{\varepsilon}(x) = [x+(x^q+x)\varepsilon]^{\sigma+2}+(x^q+x)^{\sigma}+(x^{2q}+x^2)\varepsilon+x^{q+1}+x^2,
\]
and define the  variety in $\mathrm{PG}(r,q^2)$ with affine equation:
\begin{equation}\label{eq:BT1}
\mathcal{V}^r_{\varepsilon}:\quad x_r^q + x_r = \Gamma_{\varepsilon}(x_1)+\dots+\Gamma_{\varepsilon}(x_{r-1}).
\end{equation}

Then, the point set
\[
\mathcal{H}^r_{\varepsilon} = (\mathcal{V}^r_{\varepsilon}\setminus \Sigma_{\infty}) \cup \mathcal{F},
\]
where, as in \eqref{eq:F},
\[
\mathcal{F} = \{(0,X_1,\dots,X_r)\mid X_1^{q+1}+\dots+X_{r-1}^{q+1}=0\},
\]
is a non classical quasi-Hermitian variety of $\mathrm{PG}(r,q^2)$. It is called a \emph{BT quasi-Hermitian variety} since for $r=2$, $\mathcal{H}^2_{\varepsilon}$ coincides with the Buekenhout--Tits unital.
\medskip
\subsubsection{The isomorphism issue}

The parametric nature of BM quasi-Hermitian varieties naturally
leads to equivalence questions.

In the planar case ($r=2$), the number of projectively
inequivalent BM unitals depends on the arithmetic of $q$.

\begin{theorem}[\cite{BE,E}]\label{thm:BM-unitals}
Let \(q=p^n\ge 4\) be a prime power. The number of projectively
inequivalent BM unitals in \(\PG(2,q^2)\) is
\[
\frac{1}{2n}\bigg[n_0+\sum_{k\mid n}\Phi\!\left(\frac{2n}{k}\right)p^k \bigg],
\]
where \(\Phi\) is Euler's totient function and \(n_0\) is the odd part of \(n\) if \(p>2\), while \(n_0=0\) if \(p=2\).
\end{theorem}

For \(\PG(3,q^2)\), complete classifications are available.

\begin{theorem}[\cite{AG}]\label{thm:AG-2023}
Let \(q=p^n\) with \(p\) odd. The number of projectively
inequivalent BM quasi-Hermitian varieties \(\cM_{a,b}\)
in \(\PG(3,q^2)\) is
\[
\frac{1}{n}\bigg(\sum_{k\mid n}\Phi\!\left(\frac{n}{k}\right)p^k \bigg)-2.
\]
\end{theorem}

\begin{theorem}[\cite{AGMS}]\label{thm:AGMS-2025}
For even \(q\), all BM quasi-Hermitian varieties
\(\cM_{a,b}\) in \(\PG(3,q^2)\) are projectively equivalent.
\end{theorem}

In particular, in even characteristic a rigidity phenomenon
occurs: the entire parametric family collapses into a single
projective equivalence class.

The structure of the corresponding automorphism groups
has also been determined.
For even \(q\), the stabilizer of \(M_{a,b}\) in
\(PGL_4(q^2)\) has order \(q^6(q-1)\),
and the semilinear stabilizer in
\(P\Gamma L_4(q^2)\) has order
\(q^6(q-1)\log_2 q\)
(see~\cite{AGMS} for explicit descriptions).

\medskip

We now turn to BT quasi-Hermitian varieties.
In the planar case, all BT unitals are equivalent under
the action of \(P\Gamma L(3,q^2)\).
The equivalence problem in higher dimension
has also been settled.

\begin{theorem}[\cite{AM}]
Let \(\varepsilon_1,\varepsilon_2 \in \GF{q^2}\)
satisfy \(\varepsilon_i^2+\varepsilon_i=\delta_i\)
with \(\operatorname{tr}(\delta_i)=1\), \(i=1,2\).
Then \(\mathcal{H}_{\varepsilon_1}^r\) and
\(\mathcal{H}_{\varepsilon_2}^r\) are projectively eqivalent.
\end{theorem}

For \(r=3\), the automorphism group of
\(\mathcal{H}_\varepsilon^3\) has been explicitly determined;
in particular,
\[
|Aut(\mathcal{H}_\varepsilon^3)\cap PGL_4(q^2)|=2q^3,
\quad
|Aut(\mathcal{H}_\varepsilon^3)|=4eq^3.
\]
We refer to~\cite{AA,AM} for full structural descriptions.

\medskip

This summary provides the foundational context for presenting some open problems, which are collected in the final section of the article.

It is worth emphasizing that the various known constructions
of quasi-Hermitian varieties (BM, BT, and other families)
are not projectively equivalent in general, see \cite{LS}.
Having reviewed the constructions and equivalence results,
we now turn to the Barlotti--Cofman representation,
which provides the framework for the structural analysis
carried out in the subsequent sections.

\section{ Barlotti-Cofman representation  of $\mathrm{PG}(r,q^n)$}
\label{sec:BCrep}
 We describe the Barlotti--Cofman (BC) representation  of $\mathrm{PG}(r,q^n)$ that provides an \(\mathbb{F}_q\)-linear model of \(\mathrm{PG}(r,q^n)\) inside \(\mathrm{PG}(rn,q)\), exploiting the structure of Desarguesian spreads. This representation is a powerful tool for studying varieties defined over extension fields by embedding them into a higher-dimensional space over the base field.

Our main reference for this section is
\cite[Section 2]{MP} .

Consider a family $\mathcal{S}$ of $(n-1)$-subspaces in the projective space $\PG(m-1,q)$. This $\mathcal{S}$ is called  a \textbf{$(n-1)$-spread}  of  $\PG(m-1,q)$ if its elements are disjoint and every point of $\PG(m-2,q)$ is contained in one of the subspaces of $\mathcal{S}$. An $(m-1)$-dimensional projective space admits an $(n-1)$-spread  if and only if $n$ divides $m$.

Given an $(n-1)$-spread $\mathcal{S}$ of $\Sigma = \mathrm{PG}(rn-1,q)$ and an ambient space $\Sigma^\ast = \mathrm{PG}(rn,q)$ with $\Sigma$ as a hyperplane, one can consider a $2$-$(q^{rn}, q^n, 1)$ design $D(\mathcal{S})$ whose points are those of $\Sigma^\ast \setminus \Sigma$ and whose blocks are the $n$-subspaces of $\Sigma^\ast$ meeting $\Sigma$ in an element of $\mathcal{S}$. The spread $\mathcal{S}$ is called \emph{Desarguesian} precisely when $D(\mathcal{S}) \cong \mathrm{AG}(r,q^n)$. Up to isomorphism, there is a unique Desarguesian $(n-1)$-spread of $\mathrm{PG}(rn-1,q)$.

Desarguesian spreads are the key ingredient for the \textbf{Barlotti-Cofman representation} of $\PG(r,q^n)$ in $\PG(rn,q)$ that we illustrate in the following.

Let $\mathcal{S}$ be a Desarguesian $(n-1)$-spread of $\Sigma = \mathrm{PG}(rn-1,q)$ and let $\Sigma^\ast = \mathrm{PG}(rn,q)$ with $\Sigma$ as a hyperplane. An \emph{$\mathcal{S}$-subspace} of $\Sigma$ is a subspace $X \subseteq \Sigma$ whose points are partitioned by the elements of $\mathcal{S}$.
Whenever $X$ is an $\mathcal{S}$-subspace distinct from a single spread element, the spread $\mathcal{S}$ induces on $X$ a Desarguesian $(n-1)$-spread $\mathcal{S}_X$.
In this case, the dimension of $X$ is $tn-1$, where $2 \le t \le r$.

 $\mathcal{S}$ determines a representation of $\PG(r-1,q^n)$ in $\PG(rn-1,q)$ where  the points and the $t$-subspaces of $\PG(r-1,q^n)$ correspond to the  elements of the spread and the $n(t+1)-1$ subspaces of
$\PG(rn-1,q)$; we denote with $\PG(\mathcal{S})$ this representation.
Next, define the point-line geometry
\[
\Pi_r = \Pi_r(\Sigma^\ast, \Sigma, \mathcal{S})
\]
as follows:
\begin{itemize}
    \item \emph{Points}: the points of $\Sigma^\ast \setminus \Sigma$ (called \emph{affine points}) and the elements of the spread $\mathcal{S}$;
    \item \emph{Lines}: the $n$-subspaces of $\Sigma^\ast$ intersecting $\Sigma$ in an element of $\mathcal{S}$ and the lines of $\mathrm{PG}(\mathcal{S})$;
    \item \emph{Incidences}: inherited from $\Sigma$ and $\Sigma^\ast$.
\end{itemize}
Then $\Pi_r$ is isomorphic to $\mathrm{PG}(r,q^n)$; this is the \emph{Barlotti--Cofman representation} (BC representation) of $\mathrm{PG}(r,q^n)$ in $\mathrm{PG}(rn,q)$ (an $\mathbb{F}_q$-linear representation). Moreover, a $t$-subspace of $\Pi_r$ is either:
\begin{enumerate}
    \item a $tn$-subspace of $\Sigma^\ast$ meeting $\Sigma$ in an $\mathcal{S}$-subspace of dimension $tn-1$, or
    \item a $t$-subspace of $\mathrm{PG}(\mathcal{S})$, i.e., an $\mathcal{S}$-subspace of $\Sigma$ of dimension $(t+1)n-1$.
\end{enumerate}
In particular, $\mathrm{PG}(\mathcal{S})$ is a hyperplane of $\Pi_r$.

\medskip

\subsection{BC representation of Hermitian varieties}\label{ssec:Q}
\noindent\textbf{Case $n=2$.} Assume that $\mathcal{S}$ is a line-spread of $\Sigma = \mathrm{PG}(2r-1,q)$ and $\Pi_r$ is the BC representation of $\mathrm{PG}(r,q^2)$ in $\Sigma^\ast = \mathrm{PG}(2r,q)$.

An Hermitian variety $\cH(r,q^2)$ in $\PG(r,q^2)$ that is secant to the hyperplane $\PG(\mathcal{S})$ is represented in $\Sigma^\ast$ by a parabolic quadric $Q(2r,q)$  intersecting $\Sigma$ in a union of elements of $\mathcal{S}$ that form a hyperbolic quadric $Q^+(2r,q)$ (r even) or an elliptic quadric $Q^-(2r,q)$ (r odd).
An Hermitian variety $\cH(r,q^2)$ that is tangent to $\PG(\mathcal{S})$ in a point $Y$ ($Y$ represented by an element of $\mathcal{S}$), is represented in $\Sigma^\ast$ by a quadric cone $\mathcal{H}'$ such that
\begin{itemize}
\item[Q1)] the vertex of $\mathcal{H}'$ is a point of the line $Y$;
\item[Q2)] the base is a quadric $Q^+(2r-1,q)$ ($r$ odd) or $Q^-(2r-1,q)$ ($r$ even);
\item[Q3)] $\mathcal{H}=\mathcal{H}'\cap \Sigma$ is a cone which is a union of elements of $\mathcal{S}$ with $Y$ as vertex.
\end{itemize}
 The following result holds for $r=3,4$ from \cite{MP}:
 \begin{theorem}
     Every quadric cone $\mathcal{H}'$ in $\PG(2r,q)$, with $r=3,4$, satisfying conditions Q1), Q2) and Q3) represents in $\Pi_r$ a non degenerate Hermitian variety $\cH(r,q^2)$ having $\PG(\mathcal{S})$ as tangent hyperplane.
 \end{theorem}

\subsection{ Coordinates in  BC representation of  $\mathrm{PG}(3, q^2)$}\label{subsec:expl}
We consider the case of $m=6$ and $n=2$.
We show how the coordinates of points in $\mathrm{PG}(3, q^2)$ relate to
the coordinates of the corresponding points in the BC representation
in $\mathrm{PG}(6, q)$.

 To avoid any confusion the coordinates of elements in $\PG(3,q^2)$ are denoted with capital letters $(X_0,X_1,X_2,X_3)=(J,X,Y,Z)$ and the ones for $\PG(6,q^2)$ with minuscule letters $(x_0,x_1,x_2,x_3,x_4,x_5,x_6)$.
 The hyperplane at infinity $\Sigma_{\infty}$ in $\PG(3,q^2)$ has equation $J=0$, whereas the hyperplane at infinity $\Pi_{\infty}$ in $\PG(6,q^2)$ has equation $x_0=0$.

Let $\epsilon\in \mathbb{F}_{q^2}\setminus \mathbb{F}_{q}$, such that $(1,\epsilon)$ is a basis of $\mathbb{F}_{q^2}$ viewed as vector space over $\mathbb{F}_{q}$.
The following  map
\[\psi:
(1,x_1+\epsilon x_2,x_3+\epsilon x_4,x_5+\epsilon x_6)\rightarrow(1,x_1,x_2,x_3,x_4,x_5,x_6)
\]
is a bijection between  affine points in $\PG(3,q^2)$ and affine points in $\PG(6,q)$.
We need to consider how the map $\psi$ acts on the lines of $\PG(3,q^2)$. This determines uniquely the spread element associated with a point of $\Sigma_{\infty}$ Thus, we group the  affine lines of
$\PG(3,q^2)$ into parallel classes of lines in $AG(3,q^2)$.
 The points in $\Sigma_{\infty}$ can be written in one of the following forms:
\begin{itemize}
    \item[type a)] $P=(0,1,k,h)$ with $k,h\in \mathbb{F}_{q^2}$,
    \item[type b)] $Q=(0,0,1,h)$ with $h\in \mathbb{F}_{q^2}$,
    \item[type c)] $P_{\infty}=(0,0,0,1)$.
\end{itemize}

We start with the $q$ odd case choosing  $\epsilon\in \mathbb{F}_{q^2}\setminus \mathbb{F}_{q}$ such that $\epsilon^q+\epsilon=0$ and  $\epsilon ^2=\delta \in \mathbb{F}_{q}$.

We compute the element of the (regular) spread $\mathcal S$ corresponding to a point written as in type a) that is,
$P(0,1,k,h)$,  by applying the  map $\psi$ to the affine lines of
$\PG(3,q^2)$ through this point.

An  affine line in $\PG(3,q^2)$ that passes through $P$ has parametric equations
\[
\ell_{b,c}=
\begin{cases}

    X=t\\
    Y=b+kt\\
    Z=c+ht
\end{cases}\;\;\;\;\;\; b,c\in \mathbb{F}_{q^2}.
\]

Write the unknowns $X=x_1+\epsilon x_2$, $Y=x_3+\epsilon x_4$,
$Z=x_5+\epsilon x_6$ and the coefficients $b=b_0+\epsilon b_1$, $c=c_0+\epsilon c_1$, $h=h_0+\epsilon h_1$, $k=k_0+\epsilon k_1$, $t=t_0+\epsilon t_1$.
Equating like powers of $\epsilon$ gives us
 the plane in $AG(6,q^2)$ corresponding to $\ell_{b,c}$:
\[
\pi_{b,c}=
\begin{cases}
        x_1=t_0\\
    x_2=t_1\\
    x_3=b_0+k_0t_0+\delta k_1t_1\\
    x_4=b_1+k_1t_0+k_0t_1\\
    x_5=c_0+h_0t_0+\delta h_1t_1\\
    x_6=c_1+h_1t_0+h_0t_1
\end{cases}
\]

As $b$ and $c$ vary over $\mathbb{F}_{q^2}$ every plane $\pi_{b,c}$ contains the  spread line
\[
r_{P}=
\begin{cases}
    x_0=0\\
    x_3=k_0x_1+\delta k_1x_2\\
    x_4=k_1x_1+k_0x_2\\
    x_5=h_0x_1+\delta h_1x_2\\
    x_6=h_1x_1+h_0x_2
\end{cases}
\]
which is independent from the choice of $b,c$.

A similar reasoning brings to the association
\[
Q=(0,0,1,h) \leftrightarrow
r_Q=
\begin{cases}
    x_0=0\\
    x_1=0\\
    x_2=0\\
    x_5=h_0x_3+\delta h_1x_4\\
    x_6=h_1x_3+h_0x_4
\end{cases}
\]
and \[
P_{\infty}(0,0,0,1) \leftrightarrow
r_{P_{\infty}}=
\begin{cases}
    x_0=0\\
    x_1=0\\
    x_2=0\\
    x_3=0\\
    x_4=0
\end{cases}
\]
The set $\{r_{R}\}_{R\in \Sigma_{\infty} }$ is a spread of $\Pi_{\infty}: x_0=0$.

For the even case we consider
$\epsilon \in \mathbb{F}_{q^2}\setminus \mathbb{F}_{q}$ and $\delta \in \mathbb{F}_{q}$ as in Subsubsection \ref{sssec:BT}. This brings some different computations in the association between points in $\PG(3,q^2)\cap \{X_0=0\}$ and lines in $\PG(6,q)\cap \{x_0=0\}$:

\[
P=(0,1,k,h)
\leftrightarrow
r_{P}=
\begin{cases}
    x_0=0\\
    x_3=k_0x_1+\delta k_1x_2\\
    x_4=k_1x_1+(k_0+k_1)x_2\\
    x_5=h_0x_1+\delta h_1x_2\\
    x_6=h_1x_1+(h_0+h_1)x_2
\end{cases}
\]
\[
Q=(0,0,1,k)
\leftrightarrow
r_{Q}=
\begin{cases}
    x_0=0\\
    x_1=0\\
    x_2=0\\
    x_5=h_0x_3+\delta h_1x_4\\
    x_6=h_1x_3+(h_0+h_1)x_4
\end{cases}
\]
\[
P_{\infty}=(0,0,0,1)
\leftrightarrow
r_{P_{\infty}}=
\begin{cases}
    x_0=0\\
    x_1=0\\
    x_2=0\\
    x_3=0\\
    x_4=0
\end{cases}\;\;\;.
\]

The BC representation translates varieties over \(\mathbb{F}_{q^2}\) into configurations in \(\mathrm{PG}(2r,q)\), enabling the use of classical tools from quadratic and linear geometry. In the next sections, we apply this framework to describe BM and BT quasi-Hermitian varieties in \(\mathrm{PG}(6,q)\).

\section{ BC representation of BM quasi-Hermitian varieties in odd characteristic}\label{sec:odd}

We consider the surface $\mathcal{B}_{a,b}$ in $\PG(3,q^2)$, $q$ odd, whose equation is:
\[
Z^qJ^q-ZJ^{2q-1}+a^q(X^{2q}+Y^{2q})-a(X^2+Y^2)J^{2q-2}=(b^q-b)(X^{q+1}+Y^{q+1})J^{q-1}
\]
where $a\in \mathbb{F}^*_{q^2}$ and $b\in \mathbb{F}_{q^2}\setminus \mathbb{F}_{q}$ and $4a^{q+1}+(b^q-b)^2\neq 0$.

The affine points of $\mathcal{B}_{a,b}$ correspond to the following set:

\begin{equation}\label{eq:def_Bab}
B_{a,b}=\{(1,x,y,z)| z+a(x^2+y^2)-b(x^{q+1}+y^{q+1})\in \mathbb{F}_q\}.
\end{equation}
We recall that $\mathcal{F}$ is the cone with equation
\[
\mathcal{F}:
\begin{cases}
J=0\\
X^{q+1}+Y^{q+1}=0
\end{cases}
\]
that lies in the hyperplane $\Sigma_{\infty}:\{J=0\}\subset \PG(3,q^2)$ and  set $\mathcal{M}_{a,b}$ in $\PG(3,q^2)$:
\[
\mathcal{M}_{a,b}= B_{a,b}\cup \mathcal{F}.
\]
Note that $\mathcal{M}_{a,b}$ and $\mathcal{B}_{a,b}$ share the same affine points.

We apply the Barlotti-Cofman representation of $\PG(3,q^2)$ in $\PG(6,q)$ explicitly defined in the previous section and describe the sets in $\PG(6,q)$ that correspond to $\mathcal{M}_{a,b}$ and $\mathcal{B}_{a,b}$.

 \subsection{BC Representation of $\mathcal{B}_{a,b}$}

We consider the basis $(1,\epsilon)$ of $\mathbb{F}_{q^2}/\mathbb{F}_{q}$ already introduced in Subsection \ref{subsec:expl} with $\epsilon^2=\delta \in \mathbb{F}_{q}$.
 We write
\[
z=x_5+\epsilon x_6; x=x_1+\epsilon x_2; y=x_3+\epsilon x_4; a=a_0+\epsilon a_1; b=b_0+ \epsilon b_1;
\]
where $x_i,a_i,b_i \in \mathbb{F}_{q}$, and substitute $x,y,z,a,b$ in the equation of $B_{a,b}$ given in \eqref{eq:def_Bab}:
\begin{equation}\label{eq1}
\begin{split}
 z+a(x^2+y^2)-b(x^{q+1}+y^{q+1})=&
 x_5+\epsilon  x_6+(a_0+\epsilon  a_1)((x_1+\epsilon  x_2)^2+(x_3+x_4)^2)\\&-(b_0+ \epsilon  b_1)((x_1+\epsilon  x_2)^{q+1}+(x_3+\epsilon  x_4)^{q+1}).
\end{split}
\end{equation}
We compute
\begin{align}\label{eq2}
(x_1+\epsilon  x_2)^2&=x_1^2+\delta x_2^2+2\epsilon  x_1x_2;\\
(x_1+\epsilon  x_2)^{q+1}&=(x_1^2+\delta  x_2^2+\epsilon  x_1^qx_2+\epsilon ^qx_1x_2^q)=x_1^2-\epsilon  x_2^2+(\epsilon +\epsilon ^q)x_1x_2=x_1^2-\delta x_2^2.
\end{align}

and substitute \eqref{eq2} in \eqref{eq1} obtaining  :

\[
x_5+\epsilon  x_6+(a_0+\epsilon  a_1)(x_1^2+\delta x_2^2+2\epsilon  x_1x_2+x_3^2+\delta x_4^2+2\epsilon  x_3x_4)-
\]
\[
-(b_0+\epsilon  b_1)(x_1^2-\delta  x_2^2+x_3^2-\delta  x_4^2) \in \mathbb{F}_{q}
\]
Asking that the above expression is in $\mathbb{F}_{q}$ corresponds to asking
\begin{equation}\label{eq:Babinpg6}
 x_6+a_0(2(x_1x_2+x_3x_4))
+a_1(x_1^2+x_3^2+\delta(x_2^2+x_4^2))
- b_1(x_1^2-\delta x_2^2+x_3^2-\delta x_4^2)=0.
\end{equation}
We deduce that the affine points in $\mathcal{B }_{a,b}$ correspond via BC representation to the points $(1,x_1,x_2,x_3,x_4,x_5,x_6)\in AG(6,q)$ that satisfy \eqref{eq:Babinpg6}.

From \cite{MP} we know that an Hermitian variety having $\PG(\mathcal{S})=\{J=0\}$ as a tangent hyperplane at a point $Y$ is represented by a quadric cone $\mathcal{H}'$ that satisfy Q1) Q2) Q3) from Subsection \ref{ssec:Q}. Neither $\mathcal{M}_{a,b}$ nor $\mathcal{B}_{a,b}$ are Hermitian varieties so we don't expect the three properties to hold but still we investigate which of them are true.
 Let $\mathcal{B}'$ be the representation of $\mathcal{B}_{a,b}$ via BC representation.
The equation of $\mathcal{B}'$ is
\begin{equation}\label{eq:B'}
 x_0x_6+a_0(2(x_1x_2+x_3x_4))+a_1(x_1^2+x_3^2+\delta(x_2^2+x_4^2))
- b_1(x_1^2-\delta x_2^2+x_3^2-\delta x_4^2)=0.
\end{equation}

This represents a quadratic cone in $\PG(6,q)$ with vertex $V=(0,0,0,0,0,1,0)$ on $ r_{P_{\infty}}$, satisfying Q1). The base of the cone is a quadric in $\PG(5,q)$ (with coordinates $[x_0,x_1,x_2,x_3,x_4,x_6]$) in line with Q2). The matrix associated to the base of $\mathcal{B}'$ is

\begin{equation}
\def\arraystretch{1.3}
A=
    \begin{blockarray}{(c|cc|cc|c)l}
        0&0&0&0&0&1\\
           \cline{1-6}
              0& 2(a_1-b_1)& 2a_0&0&0&0\\
   0& 2a_0& 2\delta(a_1+b_1)&0&0&0\\
           \cline{1-6}
              0&0&0& 2(a_1-b_1)& 2a_0&0\\
   0&0&0& 2a_0& 2\delta(a_1+b_1)&0\\
           \cline{1-6}
         1&0&0&0&0&0\\
    \end{blockarray}
\end{equation}

We use \cite[Cap. 1]{HT} to classify this quadric with the same notation:
\[
\det(A)=-16(\delta a_1^2- \delta b_1^2-a_0^2)^2; \;\; \alpha=-\det(A).
\]
Substituting $a=a_0+\epsilon a_1$ and $b=b_0+\epsilon b_1$ in $4a^{q+1}+(b^q-b)^2\neq 0$ we obtain $-4(\delta a_1^2- \delta b_1^2-a_0^2)\neq 0$ so $\det(A)\neq 0$ and the quadric is not degenerate.
Since $\alpha$ is a square, it follows that the base of $\mathcal{B}'$ is a hyperbolic quadric.

At last we study $\mathcal{B}=\mathcal{B}'\cap \{x_0=0\}$. We obtain a new quadratic cone with base the surface
\begin{equation}\label{eq:equationQ}
\mathcal{Q}:(a_1-b_1)(x_1^2+x_3^2)+\delta(a_1+b_1)(x_2^2+x_4^2)+2a_0(x_1x_2+x_3x_4)=0
\end{equation}
and with vertex the line $r_{P_{\infty}}\in \mathcal{S}$ (this is the line that according to our notation in Subsection \ref{subsec:expl} is associated to $(0,0,0,1)$).

\begin{proposition}\label{prop:linesB'}
The number of spread lines contained in $\mathcal{B}'$ depends on $q$. In particular:
\begin{itemize}
    \item if $q \equiv 1 \pmod{4}$, $\mathcal{B'}$ contains $2q^2+1$ spread lines;
    \item if $q \equiv 3 \pmod{4}$, $\mathcal{B'}$ contains $1$ spread line.
\end{itemize}
\end{proposition}
\begin{proof}
  We already know that $r_{P_{\infty}}$ is entirely contained in $\mathcal{B}'$ so we focus on when $r_{P}\subset \mathcal{B}'$ for $P\in \{J=0\}\setminus P_{\infty}$. According to the study carried in Subsection \ref{subsec:expl}
  we need to determine which spread lines of type
    $r_{(0,1,\alpha,t)}$ and $r_{(0,0,1,t)}$ with $\alpha,t\in \mathbb{F}_{q^2}$ are entirely contained in $\mathcal{B}'$. 
  A direct computation shows that no spread lines of type $r_{(0,0,1,t)}$ are entirely contained in $\mathcal{B}'$.
Substituting the generic point of the line
\[
r_{(0,1,\mu,t)}=
\begin{cases}
x_0=0\\
x_3=\mu_0 x_1+\delta \mu_1x_2\\
x_4=\mu_1 x_1+\mu_0x_2\\
x_5=t_0x_1+\delta t_1x_2\\
x_6=t_1x_1+t_0x_2
\end{cases}
\]
in equation \eqref{eq:equationQ} we obtain
\[
[(a_1-b_1)(1+\mu_0^2)+\delta(a_1+b_1)\mu_1^2+2a_0\mu_0\mu_1]x_1^2+\]
\[[\delta^2(a_1-b_1)\mu_1^2+\delta(a_1+b_1)(1+\mu_0^2)+2a_0 \delta\mu_0\mu_1]x_2^2+\]
\[2x_1x_2[(\delta(a_1-b_1)\mu_0\mu_1+\delta(a_1+b_1))\mu_0\mu_1+a_0(\delta\mu_1^2+\mu_0^2+1)]=0
\]
which is an identity
if and only if

\[\begin{cases}
(a_1-b_1)(1+\mu_0^2)+\delta(a_1+b_1)\mu_1^2+2a_0\mu_0\mu_1=0\\
\delta^2(a_1-b_1)\mu_1^2+\delta(a_1+b_1)(1+\mu_0^2)+2a_0 \delta\mu_0\mu_1=0\\
2\delta a_1\mu_0\mu_1+a_0(\delta\mu_1^2+\mu_0^2+1)=0
\end{cases}.
\]
Multiplying the first equation by $\delta$ and subtracting it from the second one, we obtain
\[2\delta b_1(1+\mu_0^2)-2\delta^2 b_1\mu_1^2=0. \]
That is, $\delta \mu_1^2=\mu_0^2+1$  and from the third equation of the previous system we get $2\delta a_0\mu_1^2+2\delta a_1\mu_0\mu_1=0$.

Hence if $\mu_1=0$ and $\mu_0^2=-1$ (which is possible only when $q\equiv 1 \mod 4$) 
then the $2q^2$
spread lines
$r_{(0,1,\pm\mu_0,t)}$, $t\in \mathbb{F}_{q^2}$ are contained in $\mathcal{B}'$.
Assume now that $\mu_1 \neq 0$.
Substituting $\delta \mu_1^2=\mu_0^2+1$ in the second equation of the system we obtain $\delta a_1 \mu_1+a_0 \mu_0=0$, that is $\delta a_1=-a_0(\mu_0/\mu_1)=a_0^2/a_1$ and hence
$\delta=a_0^2/a_1^2$ which is impossible.
Therefore for $q\equiv 1$ $\pmod{4}$, the cone contains $2q^2+1$ spread lines whereas for $q\equiv 3$ $\pmod{4}$ the cone contains only one spread line.

\end{proof}

As we expected, property Q3) is not satisfied as even $2q^2+1$ are not enough to cover $\mathcal{B}'$.

We now determine the BC representation of the point set at infinity
$\mathcal{B}_\infty = \mathcal{B}_{a,b} \cap \{J=0\}$.
From the defining equation of $\mathcal{B}_{a,b}$ we have
\[
\mathcal{B}_\infty=
\{J=0,\; X^{2q}+Y^{2q}=0\}
=
\{J=0,\; Y=\mu X\} \cup \{J=0,\; Y=-\mu X\},
\]
where $\mu^2=-1$ in $\mathbb{F}_{q^2}$.
Hence $\mathcal{B}_\infty$ consists of two lines meeting at $P_\infty$.

Under the BC representation, each point
$(0,1,\pm\mu,t)$ with $t \in \mathbb{F}_{q^2}$
corresponds to a spread line of type
$r(0,1,\pm\mu,t)$, together with the line $r_{P_\infty}$.
Therefore, the  points at infinity of $B_{a,b}$
are represented by $2q^2+1$ spread lines.

In particular, the containment of these spread lines in the cone $B'$
depends on $q \bmod 4$: if $q \equiv 1 \pmod 4$ then all the
$2q^2+1$ spread lines lie on $B'$,
whereas if $q \equiv 3 \pmod 4$ only the line $r_{P_\infty}$
is contained in $B'$.

We summarize the above discussion in the following theorem.

\begin{theorem}\label{thm:babodd}
    
Let $\mathcal{B}_{a,b}$ be the variety in $\PG(3,q^2)$, with $q$ odd.
In the BC representation in $\PG(6,q)$, the affine points of $\mathcal{B}_{a,b}$ correspond to the affine points of the quadratic cone $\mathcal{B'}$ with equation
 \[x_0x_6+a_0(2(x_1x_2+x_3x_4))+a_1(x_1^2+x_3^2+\delta(x_2^2+x_4^2))
- b_1(x_1^2-\delta x_2^2+x_3^2-\delta x_4^2)=0
\]
whose base is a hyperbolic quadric and whose vertex is $V=(0,0,0,0,0,1,0)$.

The BC representation of $\mathcal{B}_{a,b}$ is obtained by adjoining to these affine points the spread lines corresponding to the points of $\mathcal{B}_{a,b}\cap{J=0}$.

More precisely, the points at infinity of $\mathcal{B}_{a,b}$ correspond to $2q^2+1$ spread lines of $\mathcal{S}$. Among these lines:
\begin{itemize}
\item if $q \equiv 1 \pmod{4}$, all $2q^2+1$ spread lines are contained in $\mathcal{B'}$;
\item if $q \equiv 3 \pmod{4}$, exactly one spread line is contained in $\mathcal{B'}$.
\end{itemize}
\end{theorem}

\subsection{ BC representation of $\mathcal{M}_{ab}$ }
The BC representation of the affine points in $\mathcal{M}_{ab}$ ($=B_{ab}$) has already been studied and satisfies the equation \eqref{eq:Babinpg6}.

We continue with the study of the BC representation  of the set $\mathcal{F}$ of  points at infinity in $\mathcal{M}_{ab}$.

According to the previous section, to every point in $\Sigma_{\infty}\subset \PG(3,q^2)$ corresponds a spread line.  $\mathcal{F}\subset \PG(3,q^2)$ is the union of $q+1$ lines intersecting in $P_{\infty}=(0,0,0,1)$, the $q+1$ lines satisfy
\begin{equation}\label{eq:linesF}
\ell_{\alpha}=\begin{cases}
J=0\\
\alpha X^{}+Y^{}=0
\end{cases}
\end{equation}
where $\alpha$ ranges among the solutions of $x^{q+1}=-1$.
The line associated to $P_{\infty}$ via BC representation is
\[
r_{P_{\infty}}=
\begin{cases}
x_1=0\\
x_2=0\\
x_3=0\\
x_4=0\\
x_0=0
\end{cases}.
\]
Take now a generic point in $\ell_{\alpha}\setminus \{P_{\infty}\}$, its coordinates are $(0,1,\alpha,t)$ where $t\in \mathbb{F}_{q^2}$.
The line associated to said point has the following equations in $\{x_0=0\}\subset \PG(6,q)$:
\[
r_{(0,1,\alpha,t)}=
\begin{cases}
x_0=0\\
x_3=\alpha_0x_1+\delta \alpha_1x_2\\
x_4=\alpha_1x_1+\alpha_0x_2\\
x_5=t_0x_1+\delta t_1x_2\\
x_6=t_1x_0+t_0x_2
\end{cases},
\]
where $\alpha=\alpha_0+\epsilon\alpha_1$, $t=t_0+\epsilon t_1$ and $\alpha_i,t_i\in \mathbb{F}_{q}$.

Next we show the following proposition:
\begin{prop}\label{prop:lines_partition}
The set of lines $\{r_{P}\}_{P\in \mathcal{F}}$ forms a partition of the set

\begin{equation}\label{eq:effe}
F=
\begin{cases}
x_0=0\\
x_1^2-\delta x_2^2+x_3^2-\delta x_4^2=0
\end{cases}.
\end{equation}
\end{prop}
\begin{proof}
We start showing that each line $r_{P}$ is contained in $F$.
This is clearly true for $r_{P_{\infty}}$ so we only check it for $P_{(0,1,\alpha,t)}$. From the first two equations of
$r_{(0,1,\alpha,t)}$ we obtain
\[
x_1=-(\alpha_0x_3-\delta\alpha_1x_4);
\;\;\; x_2=-(\alpha_1x_3-\alpha_0x_4)
\]
and find out that
\[
x_1^2-\delta x_2^2=(\alpha_0^2-\delta\alpha_1^2)x_3-\delta(\alpha_0^2-\delta \alpha_1^2)x_4
\]
from the fact that $\alpha_0^2+\delta\alpha_1^2=\alpha^{q+1}=-1$ we find that
$r_{P}$ satisfies the equations in \eqref{eq:effe}.

Now we show that every point in $F$ belongs to one of the lines $r_{P}$, $P\in\mathcal{F}$ so let $Q$ be any point in $F$ with coordinates $[q_0=0,q_1,\ldots,q_6]$. From our choice of $\epsilon$ we have that $\delta$ is not a square in $\mathbb{F}_{q}$ and therefore if $q_1=q_2=0$ then $q_3=q_4=0$ and vice versa. In the case of $q_1=\ldots=q_4=0$ we have that $Q\in r-{P_{\infty}}$. In the case of $(q_1,q_2)\neq (0,0)\neq (q_3,q_4)$ we consider $Q_1=q_1+\delta q_2$ and $Q_2=q_3+\delta q_4$ and notice that the second equation in \eqref{eq:effe} implies that $Q_1^{q+1}+Q_2^{q+1}=0$. We take $\alpha=Q_1/Q_2$ and we conclude that $Q$ is an element in $r_{(0,1,\alpha,t)}$ where $t=q_5+\delta q_6$.
\end{proof}

In conclusion
\begin{theorem}\label{thm:mabodd}
Let $\mathcal{M}_{a,b}$ be the BM quasi-Hermitian variety in $\PG(3,q^2)$, with $q$ odd. 
In $\PG(6,q)$ its affine points correspond to the affine points of the quadratic cone $\mathcal{B}'$ defined by
\[
x_0x_6 
+ 2a_0(x_1x_2+x_3x_4)
+ a_1(x_1^2+x_3^2+\delta(x_2^2+x_4^2))
- b_1(x_1^2-\delta x_2^2+x_3^2-\delta x_4^2)
=0.
\]
The cone $\mathcal{B}'$ has vertex 
$
V=(0,0,0,0,0,1,0)
$
and its base is a hyperbolic quadric.

The points at infinity of $\mathcal{M}_{a,b}$ correspond to $q^3+q^2+1$ spread lines of $\PG(6,q)$, which partition the set
\[
F=
\begin{cases}
x_0=0,\\
x_1^2-\delta x_2^2+x_3^2-\delta x_4^2=0.
\end{cases}
\]
Among these lines, $2q^2+1$ are contained in $\mathcal{B}'$ if $q\equiv 1 \pmod{4}$, whereas exactly one spread line is contained in $\mathcal{B}'$ if $q\equiv 3 \pmod{4}$.
\end{theorem}

\section{ BC representation of BM quasi-Hermitian varieties in even characteristic}\label{sec:even}
In Section \ref{sec:odd} we studied the BC representation of the surfaces $\mathcal{B}_{a,b}$ and $\mathcal{M}_{a,b}$ in the case of $q$ odd, here we plan to do the same work for the case  of  $q>2$ even. Let $q$ be even throughout.

\subsection{BC representation of $\mathcal{B}_{ab}$ }

We write \[
z=x_5+\epsilon x_6; x=x_1+\epsilon x_2; y=x_3+\epsilon x_4; a=a_0+\epsilon a_1; b=b_0+ \epsilon b_1.
\]
where $x_i,a_i,b_i\in \mathbb{F}_{q}$ and, as in Subsection \ref{subsec:expl}, $\epsilon\in \mathbb{F}_{q^2}\setminus \mathbb{F}_{q}$ and $\delta$ as in \ref{sssec:BT}. We substitute $x,y,z,a,b$ in the equation for $B_{a,b}$ \eqref{eq:def_Bab} obtaining:
\[
x^2=(x_1+\epsilon x_2)^2=x_1^2+\epsilon^2 x_2^2=x_1^2+\epsilon x_2^2+\delta x_2^2;
\]
\[
x^{q+1}=(x_1+\epsilon x_2)(x_1+\epsilon x_2)^{q}=
x_1^2+(\epsilon^q+\epsilon)x_1x_2+\epsilon^{q+1}=x_1^2+\delta x_2^2+x_1x_2.
\]
and
\[
x_5+\epsilon x_6+(a_0+\epsilon a_1)(x_1^2+\epsilon x_2^2+\delta x_2^2+x_3^2+ \epsilon x_4^2+\delta x_4^2)+(b_0+\epsilon b_1)(x_1^2+\delta x_2^2+x_1x_2+x_3^2+\delta x_4^2+x_3x_4).
\]

Imposing that the latter is in $\mathbb{F}_{q}$ we have:
\begin{equation}\label{eq:puntialfini}
x_6+a_0(x_2^2+x_4^2)+a_1(x_1^2+x_2^2+\delta x_2^2+x_3^2+ x_4^2+\delta x_4^2)+b_1(x_1^2+\delta x_2^2+x_1x_2 +x_3^2+\delta x_4^2+x_3x_4)=0.
\end{equation}

In accordance with property Q1) we see that \eqref{eq:puntialfini} is the equation of a quadric cone $\mathcal{B}'$ in $\PG(6,q)$ with vertex $V=(0,0,0,0,0,1,0)$:
\[
x_0x_6+a_0(x_2^2+x_4^2)+a_1(x_1^2+x_2^2+\delta x_2^2+x_3^2+ x_4^2+\delta x_4^2)+b_1(x_1^2+\delta x_2^2+x_1x_2 +x_3^2+\delta x_4^2+x_3x_4)=0.
\]

We interpret \eqref{eq:puntialfini} as the equation of the base of the cone in $\PG(5,q)$. We study the matrix $A$ associated to the base of the cone (following \cite[Cap. 1]{HT}):

\begin{equation}
\def\arraystretch{1.3}
    \begin{blockarray}{(c|cc|cc|c)l}
        0&0&0&0&0&1\\
           \cline{1-6}
              0&2(a_1+b_1)& b_1&0&0&0\\
   0& b_1& 2(a_0+(1+\delta)a_1+\delta b_1)&0&0&0\\
           \cline{1-6}
              0&0&0&2(a_1+b_1)& b_1&0\\
   0&0&0&b_1&2(a_0+(1+\delta)a_1+\delta b_1)&0\\
           \cline{1-6}
         1&0&0&0&0&0\\
    \end{blockarray}
\end{equation}

we deduce that
\[
\det(A)=-(4(a_1+b_1)(a_0+(1+\delta)a_1+b_1\delta)-b_1^2)^2
\]
meaning that $\det(A)=-b_1^4\neq 0$. We consider a second matrix $B$ of the same size as $A$, defined by $b_{i,i}=0$ and $b_{ji}=-b_{ij}=-a_{ij}$ for $i<j$.
We obtain $\det(B)=-b_1^4$. Following \cite[Theorem 1.2]{HT} we compute
\begin{align*}
\alpha&=\frac{\det(B)-\det(A)}{4\det(B)}=\\
&=\frac{4(a_1+b_1)^2(a_0+(1+\delta)a_1+b_1\delta)^2+2((a_1+b_1)(a_0+(1+\delta)a_1+b_1\delta))b_1^2}{-b_1^2}
\end{align*}
Interpreting the formula for $\alpha$ as a rational function over $\mathbb{Z}$ in indeterminates and then specialising to $\mathbb{F}_q$ as in \cite[page 4]{HT}, we obtain $\alpha=0$. 
Since $\det(A)\neq 0$  and $\operatorname{tr}(\alpha)=0$, it follows that   the base of the cone $\mathcal{B}'$ is a non degenerate hyperbolic quadric, satisfying therefore property Q2). We now study the  points at infinity of $\mathcal{B}'$:

\begin{equation}\label{eq:even}
  \mathcal{B'}\cap \{x_0=0\}:  (a_0+a_1)(x_2^2+x_4^2)+(a_1+b_1)(x_1^2+\delta x_2^2+x_3^2+\delta x_4^2)+b_1(x_1x_2+x_3x_4)=0.
\end{equation}

Similarly to the odd case, \eqref{eq:even} the variety represents a quadric cone with vertex the line $r_{P_{\infty}}\in \mathcal{S}$.
 We study the equation of the other spread lines contained in $\mathcal{B}'$:

\begin{proposition}\label{prop:B'even}
The number of spread lines contained in $\mathcal{B}'$ for $q$ even is $q^2+1$.
\end{proposition}
\begin{proof}
From Subsection \ref{subsec:expl} we know we need to study three types of spread lines: $r_{(0,0,0,1)}$, $r_{(0,0,1,t)}$ and $r_{(0,1,\mu,t)}$ where $t,\mu\in \mathbb{F}_{q^2}$.
 We already noticed that $r_{P_{\infty}}=r_{(0,0,0,1)}$ is entirely contained in $\mathcal{B}'$. On the other hand it is an easy computation to show that no spread line of type $r_{(0,0,1,t)}$ is entirely contained in $\mathcal{B}'$.

We consider therefore the generic spread line with equation
\[
r_{(0,1,\mu_0+\epsilon \mu_1,t_0+\epsilon t_1)}=
\begin{cases}
    x_0=0\\
    x_3=\mu_0x_1+\delta \mu_1x_2\\
    x_4=\mu_1x_1+(\mu_0+\mu_1)x_2\\
    x_5=t_0x_1+\delta t_1x_2\\
    x_6=t_1x_1+(t_0+t_1)x_2
\end{cases}.
\]
We substitute $x_3$ and $x_4$ in equation \eqref{eq:even}:
\begin{align*}
[(a_0+a_1)\mu_1^2+(a_1+b_1)(1+\mu_0^2)+\delta(a_1+b_1)\mu_1^2+b_1\mu_0\mu_1]x_1^2
+&\\+[(a_0+a_1)(1+\mu_0^2+\mu_1^2)+(a_1+b_1)\delta^2\mu_1^2+\delta(a_1+b_1)(1+\mu_0^2+\mu_1^2)&\\+\delta b_1\mu_1(\mu_0+\mu_1)]x_2^2+(1+\mu_0^2+\delta\mu_1^2+\mu_0\mu_1)x_1x_2=0
\end{align*}
the above is an identity for every $(x_1,x_2)\in \PG(2,q)$ only if
\begin{equation}
\begin{cases}\label{eqsist}
    &[(a_0+a_1)\mu_1^2+(a_1+b_1)(1+\mu_0^2)+\delta(a_1+b_1)\mu_1^2+b_1\mu_0\mu_1]=0\\
    &[(a_0+a_1)(1+\mu_0^2+\mu_1^2)+(a_1+b_1)\delta^2\mu_1^2+\delta(a_1+b_1)(1+\mu_0^2+\mu_1^2)+\\&+\delta b_1\mu_1(\mu_0+\mu_1)]=0\\
   &(1+\mu_0^2+\delta\mu_1^2+\mu_0\mu_1)=0
    \end{cases}
\end{equation}
the last equation in system \eqref{eqsist} implies that
\begin{equation}\label{eq18}
1+\mu_0^2+\delta\mu_1^2=\mu_0\mu_1
\end{equation}
we substitute this in the first equation of \eqref{eqsist} and obtain
\[
(a_0+a_1)\mu_1^2+a_1\mu_0\mu_1=\mu_1[(a_0+a_1)\mu_1+a_1\mu_0]=0.
\]

We distinguish two cases.

\medskip
\noindent
\textit{Case $\mu_1=0$.} From \eqref{eq18} we obtain $1+\mu_0^2=0$, hence $\mu_0=1$. Therefore we obtain the $q^2$ spread lines of type
\[
r(0,1,1,t_0+\epsilon t_1), \qquad t_0,t_1\in\mathbb{F}_q.
\]

\medskip
\noindent
\textit{Case $\mu_1\neq 0$.} Then
\begin{equation}\label{eq19}
a_1\mu_0+(a_0+a_1)\mu_1=0.
\end{equation}

From \eqref{eq18} we also have
\[
1+\mu_0^2+\mu_1^2=\mu_0\mu_1+(\delta+1)\mu_1^2.
\]
Substituting this into the second equation of  \eqref{eqsist}, after straightforward simplifications in characteristic $2$, we obtain
\[
\bigl(a_0+a_1+\delta a_1\bigr)\mu_0\mu_1+\bigl(a_0(\delta+1)+a_1\bigr)\mu_1^2=0.
\]
Since $\mu_1\neq 0$, dividing by $\mu_1$ gives
\begin{equation}\label{eq20}
\bigl(a_0+a_1+\delta a_1\bigr)\mu_0+\bigl(a_0(\delta+1)+a_1\bigr)\mu_1=0.
\end{equation}

Equations \eqref{eq19} and \eqref{eq20} form a homogeneous linear system in $(\mu_0,\mu_1)$. For a non-trivial solution to exist, its determinant must vanish, that is,
\[
a_1\bigl(a_0(\delta+1)+a_1\bigr)+(a_0+a_1)\bigl(a_0+a_1+\delta a_1\bigr)=0.
\]
In characteristic $2$, this simplifies to
\[
a_0^2+a_0a_1+\delta a_1^2=0.
\]

If $a_1=0$, then $a_0\neq 0$ and the above equation is impossible. If $a_1\neq 0$, dividing by $a_1^2$ yields
\[
\left(\frac{a_0}{a_1}\right)^2+\frac{a_0}{a_1}+\delta=0,
\]
which has no solution in $\mathbb{F}_q$ since $\mathrm{Tr}(\delta)=1$. Hence the case $\mu_1\neq 0$ is impossible.

\end{proof}

As we expected,  property Q3) fails for $B_{a,b}$. We study now the BC representation of $\mathcal{B}_{\infty}=\mathcal{B}_{a,b}\cap \{J=0\}$. From the definition of $\mathcal{B}_{a,b}$ we know that
\[
\mathcal{B}_{a,b}\cap \{J=0\}:
\begin{cases}
J=0\\
X^{2q}+Y^{2q}=0
\end{cases}
=\{J=0;Y= X\}
\]
and hence  the points in $\mathcal{B}_{\infty}$ are the $q^2+1$ points on a line. The spread lines associated to this points are precisely the ones found in Proposition \ref{prop:B'even}.

\begin{theorem}\label{thm:babeven}
    The variety $\mathcal{B}_{a,b}$ in $\PG(3,q^2)$, $q$ even, is completely determined  by the  quadratic cone $\mathcal{B}'$  of  $\PG(6,q)$ with equation
  \[
x_0x_6
+ a_0(x_2^2+x_4^2)
+ a_1(x_1^2+x_2^2+\delta x_2^2+x_3^2+x_4^2+\delta x_4^2)
+ b_1(x_1^2+\delta x_2^2+x_1x_2+x_3^2+\delta x_4^2+x_3x_4)
=0
\]
with base a hyperbolic quadric and vertex $V=(0,0,0,0,0,1,0)$.
    The points at infinity of $\mathcal{B}_{a,b}$ are represented by $q^2+1$ lines of the spread $\mathcal{S}$ all entirely contained in $\mathcal{B}'$.
\end{theorem}

\subsection{BC Representation of $\mathcal{M}_{ab}$}\label{ssec:mabeven}
We study at last $\mathcal{F}=\mathcal{M}_{a,b}\cap \{J=0\}$, this is very similar to the study of $\mathcal{F}$ carried in Section \ref{sec:odd} which does not use the fact that $q$ is odd. We have that
$\mathcal{F}$ is the union of $q+1$ lines with equation \ref{eq:linesF}
meeting at $P_{\infty}$.
 The lines associated to each point in $\mathcal{F}$ other than $r_{P_{\infty}}$ have equation
\begin{equation}\label{eq:pointsF}
r_{(0,\alpha,1,t)}=
\begin{cases}
x_0=0\\
x_3=\alpha_0x_1+\delta \alpha_1x_2\\
x_4=\alpha_1x_1+\alpha_0x_2+\alpha_1x_2\\
x_5=t_0x_1+\delta t_1x_2\\
x_6=t_1x_1+t_0x_2+t_1x_2
\end{cases}
\end{equation}
where $\alpha=\alpha_0+\epsilon\alpha_1$, $t=t_0+\epsilon t_1$ and $\alpha_i,t_i\in \mathbb{F}_{q}$.

In conclusion:

\begin{theorem}\label{thm:mabeven}
Let $\mathcal{M}_{a,b}$ be the BM quasi-Hermitian variety in $\PG(3,q^2)$, with $q$ even. 
In $\PG(6,q)$, the affine points of  $\mathcal{M}_{a,b}$ correspond to  the affine points of the quadratic cone $\mathcal{B}'$ defined by
\[
x_0x_6
+ a_0(x_2^2+x_4^2)
+ a_1(x_1^2+x_2^2+\delta x_2^2+x_3^2+x_4^2+\delta x_4^2)
+ b_1(x_1^2+\delta x_2^2+x_1x_2+x_3^2+\delta x_4^2+x_3x_4)
=0.
\]
The cone $\mathcal{B}'$ has vertex 
$V=(0,0,0,0,0,1,0)$,
and its base is a hyperbolic quadric.

The points  at infinity of $\mathcal{M}_{a,b}$ are represented by the $q^3+q^2+1$ lines of the spread $\mathcal{S}$ given in~\eqref{eq:pointsF}. 
Among these lines, $q^2+1$ are contained in $\mathcal{B}'$.

\end{theorem}

\section{ BC representation of  BT quasi-Hermitian varieties}\label{sec:BT}
Throughout this section we use the same objects and notation introduced in Subsection \ref{sssec:BT}.

We study the BC representation of $\mathcal{H}^3_{\varepsilon}=(\mathcal{V}^3_{\varepsilon}\setminus \Sigma_{\infty})\cup \mathcal{F}$ starting from its affine point.  We substitute
\[
Z=x_5+\epsilon x_6; X=x_1+\epsilon x_2; Y=x_3+\epsilon x_4;
\]
in equation \eqref{eq:BT1} for $r=3$, obtaining the equation of the algebraic variety:
\begin{equation}\label{eq:coneBT}
\mathcal{C}^3_{\varepsilon}:x_6=x_1^{\sigma+2}+x_1x_2+x_2^\sigma+x_3^{\sigma+2}+x_3x_4+x_4^\sigma
\end{equation}
of $\PG(6,q)$.

\noindent
We determine the points at infinity  of  $\mathcal{C}^3_{\varepsilon}$ in $\PG(6,q)$ that is, the solutions of
\[\begin{cases}
x_0^{\sigma+1}x_6=x_1^{\sigma+2}+x_0^{\sigma}x_1x_2+x_0^2x_2^{\sigma}+x_3^{\sigma+2}+x_0^{\sigma}x_3x_4+x_0^2x_4^\sigma\\
x_0=0
\end{cases}
\]
Hence, we  solve
\[
\left(\frac{x_1}{x_3}\right)^{\sigma+2}=1 \text{ in } \{x_0=0\}.
\]
This depends on the study of $gcd(2^{\frac{e+1}{2}}+2,q-1)$ where, we recall, $x^{\sigma+2}=x^{2^{\frac{e+1}{2}}+2}$ and $q=2^e$. It turns out that the $gcd$ is always $1$, therefore \[\mathcal{C}^3_{\varepsilon}\cap \{x_0=0\}=\{x_0=0,x_1=x_3\}.\]
 The following proposition provides a cone-like description of $\mathcal{C}^3_\varepsilon$. 
\begin{proposition}
Let $\mathcal{C}^3_\varepsilon$ be the hypersurface of $\PG(6,q)$ defined by
\[
x_6 = x_1^{\sigma+2} + x_1x_2 + x_2^\sigma + x_3^{\sigma+2} + x_3x_4 + x_4^\sigma.
\]
Let $V=(0,0,0,0,0,1,0)$ and let $
\mathcal{C}^{\prime 3}_\varepsilon = \mathcal{C}^3_\varepsilon \cap \{x_5=0\}.
$
Then, set-theoretically,
\[
\mathcal{C}^3_\varepsilon = \bigcup_{P \in \mathcal{C}^{\prime 3}_\varepsilon} \langle V,P\rangle,
\]
that is, $\mathcal{C}^3_\varepsilon$ is the union of the lines joining $V$ to the points of $\mathcal{C}^{\prime 3}_\varepsilon$.
\end{proposition}

\begin{proof}
Let
\[
U=\bigcup_{P\in \mathcal{C}^{\prime 3}_\varepsilon}\langle V,P\rangle.
\]
Since the defining equation of $\mathcal{C}^3_\varepsilon$ does not depend on $x_5$, every line
\(\langle V,P\rangle\), with \(P\in \mathcal{C}^{\prime 3}_\varepsilon\), is contained in \(\mathcal{C}^3_\varepsilon\). Hence
\[
U\subseteq \mathcal{C}^3_\varepsilon.
\]
Moreover, distinct points $P\in \mathcal{C}^{\prime 3}_\varepsilon$ define distinct lines $\langle V,P\rangle$, and any two such lines meet only in $V$. Hence
\[
|U| = q\,|\mathcal{C}^{\prime 3}_\varepsilon| + 1.
\]
Since $|\mathcal{C}^{\prime 3}_\varepsilon|=q^4+q^3+q^2+q+1$, we obtain
\[
|U| = q^5+q^4+q^3+q^2+q+1 = |\mathcal{C}^3_\varepsilon|.
\]
Therefore $U=\mathcal{C}^3_\varepsilon$.
\end{proof}

\begin{remark}
In contrast with the quadratic case, the hypersurface $\mathcal{C}^3_\varepsilon$ is not a cone in the classical sense. Although it can be described as the union of lines through a fixed point, its singular locus is not a point.
\end{remark}

\begin{proposition}
The only spread lines contained in $\mathcal{C}^3_\varepsilon$ are those of type $r(0,1,1,h)$ with $h \in \mathbb{F}_{q^2}$ and the line $r_{P_{\infty}}$.
\end{proposition}

\begin{proof}
We determine which spread lines $r_P$ are entirely contained in $\mathcal{C}^3_\varepsilon$.
It is immediate to check that the line $r_{P_\infty}$ is contained in $\mathcal{C}^3_\varepsilon$. Moreover, for points $P=(0,0,1,k)$, a direct computation shows that the corresponding lines $r_P$ are not contained in $\mathcal{C}^3_\varepsilon$.

We therefore consider spread lines corresponding to points of the form $(0,1,k,h)$, with $k,h \in \mathbb{F}_{q^2}$. Using the notation introduced in Subsection 4.2, such a line has equations
\[
\begin{cases}
x_0 = 0,\\
x_3 = k_0 x_1 + \delta k_1 x_2,\\
x_4 = k_0 x_2 + k_1(x_1 + x_2),\\
x_5 = h_0 x_1 + \delta h_1 x_2,\\
x_6 = h_1 x_1 + (h_0 + h_1)x_2.
\end{cases}
\]
We intersect this line with $\mathcal{C}^3_\varepsilon \cap \{x_0=0\}$, whose equation is
$
x_1^{\sigma+2} + x_3^{\sigma+2} = 0$.
Substituting the expression for $x_3$, we obtain
\[
(k_0 x_1 + \delta k_1 x_2)^{\sigma+2} + x_1^{\sigma+2} = 0.
\]
Assume first that $x_2 \neq 0$ and set $X = x_1/x_2$. Dividing by $x_2^{\sigma+2}$, we obtain
\[
(k_0 X + \delta k_1)^{\sigma+2} + X^{\sigma+2} = 0.
\]
Expanding, this yields a polynomial equation in $X$:
\[
(k_0^{\sigma+2} + 1)X^{\sigma+2}
+ \delta^2 k_0^\sigma k_1^2 X^\sigma
+ \delta^\sigma k_0^2 k_1^\sigma X^2
+ (\delta k_1)^{\sigma+2} = 0. \tag{*}
\]
For the entire line to be contained in $\mathcal{C}^3_\varepsilon$, this equation must hold for all $X \in \mathbb{F}_q$, hence all coefficients must vanish.
From the coefficient of $X^{\sigma+2}$ we obtain
$
k_0^{\sigma+2} + 1 = 0$,
which implies $k_0 = 1$.
Substituting this into the remaining coefficients, we obtain
\[
\delta^2 k_1^2 = 0, \qquad \delta^\sigma k_1^\sigma = 0,
\]
hence $k_1 = 0$. Therefore $k = k_0 + \epsilon k_1 = 1$.

Finally, if $x_2 = 0$, then $x_1 \neq 0$ and the defining equation reduces to
\[
x_1^{\sigma+2} = 0,
\]
which again implies $k_0 = 1$ and $k_1 = 0$.

We conclude that the only spread lines contained in $\mathcal{C}^3_\varepsilon$ are, apart from $r_{P_\infty}$, those of type $r(0,1,1,h)$, as claimed.
\end{proof}

\subsection{BC Representation of $\mathcal{V}^3_{\epsilon} $}

From \cite{AGLS} we have
\[
\mathcal{V}^3_\varepsilon \cap \{J = 0\} =
\begin{cases}
\ell_0 & \text{if } e \equiv 1 \pmod{4},\\
\ell_0 \cup \ell_1 \cup \ell_2 & \text{if } e \equiv 3 \pmod{4},
\end{cases}
\]
where $\ell_0 : J = X + Y = 0$ and $\ell_i : J = c_i X + Y = 0$, for $i=1,2$, with $c_0=1$ and $c_1,c_2$ the $(2^{\frac{e-1}{2}}+1)$-th roots of unity in $\mathbb{F}_{q^2}$.

\medskip
\noindent
\textit{Case $e \equiv 1 \pmod{4}$.}
In this case $\mathcal{V}^3_\varepsilon \cap \{J = 0\}$ is the line $\ell_0$. Under the BC representation, its points correspond to $q^2+1$ spread lines of type
\[
r(0,1,1,h), \qquad h \in \mathbb{F}_{q^2},
\]
which are precisely the spread lines contained in $\mathcal{C}^3_\varepsilon \cap \{x_0=0\}$.

\medskip
\noindent
\textit{Case $e \equiv 3 \pmod{4}$.}
In this case $\mathcal{V}^3_\varepsilon \cap \{J = 0\}$ consists of three lines $\ell_0,\ell_1,\ell_2$. Their BC representation is given by the corresponding spread lines.

The points of $\ell_0$ yield the $q^2+1$ spread lines $r(0,1,1,h)$, while each of the lines $\ell_1$ and $\ell_2$ contributes $q^2$ spread lines of type $r(0,1,c_i,h)$, with $h \in \mathbb{F}_{q^2}$.

Therefore, the BC representation of $\mathcal{V}^3_{\epsilon}\cap \{J=0\}$  consists of $3q^2+1$ spread lines.

\begin{theorem}\label{thm:BT-quasiHermitian-cone}
Let $\mathcal{V}^3_{\varepsilon}$ be the variety in $\PG(3,q^2)$ with equation~\eqref{eq:BT1}.
In the BC representation in $\PG(6,q)$, the affine points of $\mathcal{V}^3_{\varepsilon}$ correspond to the affine points of the algebraic variety
\[
\cC^3_{\varepsilon} :
x_0^{\sigma+1}x_6
= x_1^{\sigma+2}
+ x_0^{\sigma}x_1x_2
+ x_0^2x_2^{\sigma}
+ x_3^{\sigma+2}
+ x_0^{\sigma}x_3x_4
+ x_0^2x_4^{\sigma}.
\]

The BC representation of $\mathcal{V}^3_{\varepsilon}$ is obtained by adjoining to these affine points the spread lines corresponding to the points of $\mathcal{V}^3_{\varepsilon}\cap{J=0}$.

More precisely:
\begin{itemize}
\item if $e \equiv 1 \pmod{4}$, the BC representation 
is completely determined by $\mathcal{C}^3_\varepsilon$, since all spread lines at infinity are contained in it;

\item if $e \equiv 3 \pmod{4}$, it consists of the points of $\mathcal{C}^3_{\varepsilon}$ together with $2q^2$ additional spread lines corresponding to the points of $(\ell_1 \cup \ell_2)\setminus \ell_0$.
\end{itemize}
\end{theorem}

\subsection{BC Representation of $\mathcal{H}^3_{\varepsilon}$}

By definition  of the BT quasi-Hermitian variety, $\mathcal{H}^3_{\varepsilon}\cap \{J=0\}=\mathcal{F}$ where, we recall, $\mathcal{F}=\{(0,X_1,\dots,X_r)\mid X_1^{q+1}+\dots+X_{r-1}^{q+1}=0\}$. We know that $\mathcal{F}\subset \PG(3,q^2)$ is the union of $q+1$ lines \[\ell_{\alpha}=\{(0,X,Y,Z)\in \PG(3,q^2)|Y=\alpha X\}\] where $\alpha$ ranges among the solutions of $x^{q+1}=-1$ in $\mathbb{F}_{q}$. The $q+1$ lines $\ell_{\alpha}$ meet at $P_{\infty}$, its resulting BC representation is given by $q^3+q^2+1$ spread lines of type $r_{P_{\infty}}$ and type $r_{(0,1,\alpha,h)}$, $h\in \mathbb{F}_{q}$.

We conclude the following about the BC representation of the $BT$ quasi Hermitian variety:

\begin{theorem}\label{thm:BTH}
In the BC representation in $\PG(6,q)$, the affine points of the BT quasi-Hermitian variety $\mathcal{H}^3_{\varepsilon}$ correspond to the affine points of the hypersurface $\mathcal{C}^3_{\varepsilon}$ of equation
\[
\mathcal{C}^3_{\varepsilon} :
x_0^{\sigma+1}x_6
= x_1^{\sigma+2}
+ x_0^{\sigma}x_1x_2
+ x_0^2x_2^{\sigma}
+ x_3^{\sigma+2}
+ x_0^{\sigma}x_3x_4
+ x_0^2x_4^{\sigma}.
\]
Equivalently, they correspond to the affine points of the union of the lines joining $V=(0,0,0,0,0,1,0)$ to the points of $\mathcal{C'}^3_{\varepsilon}=\mathcal{C}^3_{\varepsilon}\cap \{x_5=0\}$.

The points at infinity of $\mathcal{H}^3_{\varepsilon}$ are represented by
$q^3+q^2+1$ lines of the spread $\mathcal{S}$ given in~\eqref{eq:pointsF}.
Among these lines, exactly $q^2+1$ are contained in $\cC^3_{\varepsilon}$.
\end{theorem}
\textbf{Remark.} By the results above, in the Barlotti--Cofman representation in $\PG(6,q)$, the cone associated with the algebraic varieties under consideration does not, in general, encode all the information coming from the points at infinity.

More precisely, for both the varieties $\mathcal{B}_{a,b}$ and $\mathcal{V}_{\varepsilon}$, the lines of the spread corresponding to the points at infinity are not, in general, all contained in the associated cone, and their behaviour depends on the type of variety and on arithmetic conditions on $q$.
In particular, the cone alone does not determine the full Barlotti--Cofman representation, and the configuration of spread lines at infinity has to be regarded as an additional geometric invariant. This highlights the intrinsic twofold nature of the representation, where both the affine and the infinite components play a fundamental role.
\section{Open Problems}\label{sec:open}

\begin{itemize}

\item \textbf{BC characterization in dimension 3.}
Determine whether, for \(r=3\), the isomorphism class (projective
equivalence class) of a BM or BT quasi-Hermitian variety is
uniquely determined by the pair formed by its BC affine cone in
\(\PG(6,q)\) and the configuration of spread elements in the
section at infinity. In other words: does the BC cone plus its
infinite spread-data determine the variety up to projectivity?

\item \textbf{Extension to higher dimension.}
Extend the explicit BC constructions obtained here for
\(\mathrm{PG}(3,q^2)\) to \(\mathrm{PG}(r,q^2)\) with \(r\ge 4\).
Which features of the \(r=3\) picture (e.g. quadratic vs
non-quadratic nature of the affine cone, arithmetic dependence of
the number of spread elements contained at infinity) persist in
higher dimension, and which do not?

\item \textbf{Invariants from spread-configuration.}
Investigate whether numerical and configurational data of the spread
elements contained in the section at infinity (e.g. number of spread
lines, incidence graph, orbit decomposition under stabilizers) yield
projective invariants that distinguish inequivalent quasi-Hermitian
varieties for \(r=3\), and then for general \(r\).

\item\textbf{Automorphism groups.}
Complete the determination of the full automorphism group of BM  quasi-Hermitian varieties for $r=3$ and $q$ odd and for $r\geq 4$ and arbitrary $q$.
Complete the determination of the full automorphism group of BT quasi-Hermitian varieties in \(\mathrm{PG}(r,q^2)\) for all
\(r\ge 4\) and  \(q\) even. In parallel, study how the automorphism
groups act on the corresponding BC cones and on the spread-configuration at infinity.

\end{itemize}

\vspace{0.5cm}
\noindent
\begin{minipage}{0.48\textwidth}
Angela Aguglia\\
Dipartimento di Meccanica, Matematica e Management,\\
Politecnico di Bari,\\
Via Orabona, 4-70125 Bari, Italy\\
\texttt{angela.aguglia@poliba.it}\
\end{minipage}
\hfill
\begin{minipage}{0.48\textwidth}
Viola Siconolfi\\
Dipartimento di Meccanica, Matematica e Management,\\
Politecnico di Bari,\\
Via Orabona, 4-70125 Bari, Italy\\
\texttt{viola.siconolfi@poliba.it}\
\end{minipage}


\begin{thebibliography}{999}

\bibitem{AA}
{\bf A.~Aguglia.}
{\em Quasi-Hermitian varieties in PG(r, q$^2$), q even},
Contributions to Discrete Mathematics, 8(1):31--37, 2013.

\bibitem{ACK}
{\bf A.~Aguglia, A.~Cossidente, G.~Korchm\'aros.}
{\em On quasi-Hermitian varieties}, 
Journal of Combinatorial Designs, 20(10):433--447, 2012.

\bibitem{AG}
{\bf A.~Aguglia, L.~Giuzzi.}
{\em On the equivalence of certain quasi-Hermitian varieties},
Journal of Combinatorial Designs, pp.~1--15, 2022.


\bibitem{AGLS}
{\bf A.~Aguglia, G.~Giuzzi, G.~Longobardi, V.~Siconolfi.}
{Minimal codes from hypersurfaces in even characteristic}, arXiv 2502.02278, 2025

\bibitem{AGMS}
{\bf A.~Aguglia, G.~Giuzzi, A.~Montinaro ,V.~Siconolfi.}
{On Quasi‐Hermitian Varieties in Even Characteristic and
Related Orthogonal Arrays},Journal of Combinatorial Designs, Volume 33,  109-122, 2025.

\bibitem{AM}
{\bf A.~Aguglia, A.~Montinaro.}
{On a class of quasi-Hermitian surfaces in even characteristic}, arXiv:2508.03907,2025.


\bibitem{BE}
{\bf R.D.~Baker, G.L.~Ebert}
{On Buekenhout-Metz unitals of odd order}, Journal of Combinatorial Theory, Series A, Volume 60, Issue 1, 67-84, 1992.

\bibitem{BarE}
{\bf S.~Barwick, G.~Ebert}
{Unitals in Projective Planes}, Springer Monographs in Mathematics, Springer New York, NY, 2008.


\bibitem{CK}
{\bf R.~Calderbank, W.~M.~Kantor.}
{\em The geometry of two-weight codes},
Bull. London Math. Soc., 18(2):97--122, 1986.


\bibitem{DHOP}
{\bf F.~De Clerck, N.~Hamilton, C.~O'Keefe, T.~Penttila.}
{\em Quasi-quadrics and related structures},
Australasian J. Combin., 22:151--166, 2000.

\bibitem{D}
{\bf P.~Delsarte.}
{\em Weights of linear codes and strongly regular normed spaces},
Discr. Math., Volume 23, Issue 5,  407-438, 1973.

\bibitem{DS}
{\bf S.~De Winter, J.~Schillewaert.}
{\em A note on quasi-Hermitian varieties and singular quasiquadrics},
Bull. Belg. Math. Soc. Simon Stevin, 17(5):911--918, 2010.

\bibitem{E2} {\bf G. L. Ebert},
   {\em Incidence structures arising from Hermitian curves},
  Ars Combinatoria,
  vol. 18, 161--175,
  1984.


\bibitem{E}
{\bf G. L.~Ebert.}
{\em On Buekenhout-Metz unitals of even order},
Europ. Journ. of Comb.,Volume 13, Issue 2, 109--117, 1992.

\bibitem{JH}
{\bf J.W.P. Hirschfeld},
Projective Geometries over Finite Fields,2nd  edition, Clarendon Press,1998.


\bibitem{HT}
{\bf J.~W.~P. Hirschfeld and J.~A. Thas.}
{General Galois Geometries},
Springer Monographs in Mathematics, Springer-Verlag London, 2016.

\bibitem{HKT} {\bf J. W. P. Hirschfeld, G. Korchm{\'a}ros, F. Torres.}, Algebraic Curves over a Finite Field, Princeton University Press, 2008.


\bibitem{LLP}
{\bf M.~Lavrauw, S.~Lia, F.~Pavese.}
{\em On the geometry of the Hermitian Veronese curve and its quasi-Hermitian surfaces},
Discrete Math., 346(10):113582, 2023.

\bibitem{LS}
{\bf  S.~Lia, J.~Sheekey.}
{\em On the geometry of tensor products over finite fields},
arXiv:2311.17896, 2023.

\bibitem{MP}{\bf G.~Marino, O.~Polverino.} Ovoidal blocking sets and maximal partial ovoids
 of Hermitian varieties, \emph{Des. Codes Cryptogr.} \textbf{56} (2010), 115--130.


\bibitem{PF}
{\bf F.~Pavese.}
{\em Geometric constructions of two-character sets},
Discrete Math., 338(3):202--208, 2015.

\bibitem{SchillewaertVandeVoorde2022}
{\bf J.~Schillewaert, G.~Van de Voorde.}
{\em Quasi-polar spaces},
Comb. Theory, 2(3):Paper No. 5, 32 pp., 2022.

\bibitem{Segre1965}
{\bf B.~Segre.}
{\em Forme e geometrie hermitiane, con particolare riguardo al caso finito},
Ann. Mat. Pura Appl. (4), 70:1--201, 1965.







\end{thebibliography}
\end{document}